\documentclass[12pt]{amsart}
\usepackage{SSdefn}
\usepackage{comment}

\newcommand{\DOI}[1]{\href{http://doi.org/#1}{\color{purple}{\tiny\tt DOI:#1}}}
\newcommand{\arxiv}[1]{\href{http://arxiv.org/abs/#1}{{\tiny\tt arXiv:#1}}}

\author{Steven V Sam}
\address{Department of Mathematics, University of California San Diego, La Jolla, CA}
\email{\href{mailto:ssam@math.ucsd.edu}{ssam@math.ucsd.edu}}
\urladdr{\url{http://math.ucsd.edu/~ssam/}}
\thanks{SS was supported by NSF grant DMS-1812462.}

\author{Andrew Snowden}
\address{Department of Mathematics, University of Michigan, Ann Arbor, MI}
\email{\href{mailto:asnowden@umich.edu}{asnowden@umich.edu}}
\urladdr{\url{http://www-personal.umich.edu/~asnowden/}}
\thanks{AS was supported by NSF grants DMS-1453893.}

\title[Projective dimension over twisted commutative algebras]{A note on projective dimension over\\ twisted commutative algebras}

\date{April 4, 2025}

\begin{document}

\maketitle

\begin{abstract}
Let $M$ be a finitely generated module over a free twisted commutative algebra $A$ that is finitely generated in degree one. We show that the projective dimension of $M(\bC^n)$ as an $A(\bC^n)$-module is eventually linear as a function of $n$. This confirms a conjecture of Le, Nagel, Nguyen, and R\"omer for a special class of modules.
\end{abstract}

\section{Introduction} \label{s:intro}

Fix a positive integer $d$ and let $A=\bC[x_{i,j} \mid 1 \le i \le d,\ 1 \le j]$ be the infinite variable polynomial ring. One can picture the variables as the entries of a $d \times \infty$ matrix. The ring $A$ is obviously not noetherian, but it is known to be \emph{equivariantly noetherian} with respect to the infinite symmetric group $\fS$ or the infinite general linear group $\GL$; this means that the ascending chain condition holds for invariant ideals. The noetherian result for $\fS$ was proved by Cohen \cite{Cohen}. The noetherian result for $\GL$ follows from this, but also admits a direct (and easier) proof \cite[\S 9.1.6]{expos}.

Let $M$ be a module for $A$ that is equivariant with respect to $\fS$ or $\GL$. We also assume that $M$ is a polynomial representation of $\GL$ and that it is finitely generated in the equivariant sense. Taking invariants under an appropriate subgroup (namely, the general linear group of the subspace spanned by the standard basis vectors $e_i$ for $i > n$), one obtains a module $M_n$ over the finite variable polynomial ring $A_n=\bC[x_{i,j} \mid 1 \le i \le d,\ 1 \le j \le n]$. Given the above noetherian results, one might hope that this sequence of modules is well-behaved.

In the case of the symmetric group (and where $M$ is a homogeneous ideal of $A$), this has been investigated by Le, Nagel, Nguyen, and R\"omer. In \cite[Theorem~7.8]{NagelRoemer}, the authors show that the Hilbert series of $M_n$ behaves in a regular manner as $n$ varies: the generating function of this sequence of rational functions is itself a rational function in two variables. As a consequence, they show that the Krull dimension (in the classical sense, i.e., does not make use of the $\GL_n$-action) of $A_n/M_n$ is eventually linear \cite[Theorem~7.10]{NagelRoemer}. To translate to their notation, we take the filtered ideal $M_1 \subseteq M_2 \subseteq \cdots$. In \cite[Conjecture~1.1]{LNNR1}, the authors conjecture that the Castelnuovo--Mumford regularity of $M_n$ is eventually linear, and in \cite[Conjecture~1.3]{LNNR2} they conjecture the same for projective dimension.

In this paper, we consider the case of the general linear group. Since $\fS$ is a rather small subgroup of $\GL$, it follows that $\GL$-equivariant modules are much more constrained than $\fS$-equivariant modules. Unsurprisingly, many of the above results were previously known in the $\GL$-case: for instance, very precise results are known on the Hilbert series, and it is known that regularity is eventually constant; see \cite{NSS, symc1, symu1, regtca, hilbert}. The main result of this paper (Theorem~\ref{thm:depth}) shows that the projective dimension of $M$ is eventually linear. This confirms the conjecture of \cite{LNNR2} in the $\GL$ case. The key tools are the structure theory for modules developed in \cite{symu1}.

\section{Set-up}

We work over the complex numbers. We assume general familiarity with Young diagrams, polynomial representations, polynomial functors, and Schur functors (denoted by $\bS_\lambda$ where $\lambda$ is an integer partition), and refer to \cite{expos} for the relevant background information and detailed references. We recall that a polynomial functor is a functor $F$ from the category of vector spaces to itself such that the induced functions
\[
  \hom(V,W) \to \hom(F(V), F(W))
\]
can be described by polynomial functions for all vector spaces $V$ and $W$.

Let $\bV=\bigcup_{n \ge 1} \bC^n$ and let $\GL=\bigcup_{n \ge 1} \GL_n$. Let $\Rep^{\pol}(\GL)$ be the category of polynomial representations of $\GL$. This is equivalent to the category of polynomial functors, and we freely pass between the two points of view. The simple objects of $\Rep^{\pol}(\GL)$ are given by $\bS_\lambda(\bV)$ as $\lambda$ ranges over all partitions.

A {\bf twisted commutative algebra (tca)} is a commutative algebra object in $\Rep^{\pol}(\GL)$. Fix a $d$-dimensional vector space $E$, and put
\begin{displaymath}
A=\Sym(\bV \otimes E).
\end{displaymath}
This is a tca. It is the same ring introduced in \S \ref{s:intro}, but written in a coordinate-free manner.

By an {\bf $A$-module} we always mean a module object for $A$ in $\Rep^{\pol}(\GL)$. Explicitly, this is a module in the ordinary sense equipped with a compatible action of $\GL$ under which it forms a polynomial representation. We say that $M$ is finitely generated if there is a finite set $S$ such that the smallest $\GL$-invariant $A$-submodule of $M$ which contains $S$ is $M$ itself. Suppose that $M$ is an $A$-module. Treating $M$ and $A$ as polynomial functors, $M(\bC^n)$ is an $A(\bC^n)$-module; note that $A(\bC^n)=\Sym(\bC^n \otimes E)$ is a finite variable polynomial ring. These are the objects $M_n$ and $A_n$ from \S \ref{s:intro}.

We say that a function $f \colon \bN \to \bN$ is {\bf eventually linear} (here $\bN$ denotes the set of non-negative integers) if there exists $a \in \bN$ and $b \in \bZ$ such that $f(n)=an+b$ for all $n \gg 0$; we then call $a$ the {\bf slope} of $f$.

\section{The key technical result}

For a polynomial representation $M$ of $\GL$, we let $\gamma_M(n)$ or $\gamma(M; n)$ be the maximum size of a partition $\lambda$ with at most $n$ columns (i.e., $\lambda_1 \le n$) such that $\bS_\lambda(\bV)$ appears with nonzero multiplicity in the irreducible decomposition of $M$. The following is the key technical result we need to prove our main theorem:

\begin{theorem} \label{thm:gamma}
If $M$ is a finitely generated $A$-module then $\gamma_M$ is eventually linear with slope at most $d$.
\end{theorem}

\begin{example} \label{ex:gamma}
Let $M=A/\fa_r$ be the coordinate ring of the rank $\le r$ matrices in $E \otimes \bV$. Suppose that $\min(n,d) \ge r$. The Cauchy identity gives the decomposition
\begin{displaymath}
M(\bC^n) = \bigoplus_{\ell(\lambda) \le r} \bS_\lambda(E) \otimes \bS_\lambda(\bC^n)
\end{displaymath}
where the sum is over all partitions with at most $r$ many parts. Hence $\gamma_M(n) = rn$.
\end{example}

It is possible to give an elementary proof of Theorem~\ref{thm:gamma} (see Remark~\ref{rmk:elem}), but we will give a more conceptual proof based on the structure theory of $A$-modules from \cite{symu1}. We define the {\bf formal character} of a polynomial representation $M$ of $\GL$, denoted $\Theta_M$, to be the formal series $\sum_{\lambda} m_{\lambda} s_{\lambda}$, where the sum is over partitions, $m_{\lambda}$ is the multiplicity of $\bS_{\lambda}(\bV)$ in $M$, and $s_{\lambda}$ is a formal symbol. Note that we can read off $\gamma_M$ from $\Theta_M$.

Let $\fa_r \subset A$ be the determinantal ideal, as in Example~\ref{ex:gamma}. Let $\Mod_{A, \le r}$ be the category of modules (set-theoretically) supported on $V(\fa_r)$ (for an ideal $I$, we use $V(I)$ to denote its vanishing locus). In other words, $\Mod_{A, \le r}$ consists of modules $M$ such that for every element $x \in M$, there exists $n(x)$ such that $\fa_r^{n(x)} x = 0$. In particular, $\Mod_{A, \le r}$ is closed under extensions and taking submodules and quotient modules, so is a Serre subcategory of $\Mod_A$, and can define
\[
  \Mod_{A,>r}=\Mod_A/\Mod_{A, \le r}
\]
to be the Serre quotient category.  Let
\[
  T_{>r} \colon \Mod_A \to \Mod_{A,>r}
\]
be the quotient functor, let $S_{>r}$ be its right adjoint, and let $\Sigma_{>r} = S_{>r} \circ T_{>r}$ be the saturation functor. Also let
\[
  \Gamma_{\le r} \colon \Mod_A \to \Mod_{A,\le r}
\]
be the functor assigning to a module its maximal submodule supported on $V(\fa_r)$. By \cite[Theorem~6.10]{symu1}, $\rR \Sigma_{>r}$ and $\rR \Gamma_{\le r}$ preserve the finitely generated bounded derived categories.

Let $\rD(A)_{\le r}$, resp.\ $\rD(A)_{>r}$, be the full subcategories of the derived category $\rD(A)$ spanned by modules $M$ with $\rR \Sigma_{>r}(M)=0$, resp.\ $\rR \Gamma_{\le r}(M)=0$. We also use $\rD(A)_{\ge r+1}$ to denote $\rD(A)_{>r}$. Set
\[
  \rD(A)_r = \rD(A)_{\le r} \cap \rD(A)_{\ge r}.
\]
Then $\rD(A)$ admits a semi-orthogonal decomposition into the $\rD(A)_0, \ldots, \rD(A)_d$. This holds for the finitely generated bounded derived categories too \cite[\S 4]{symu1}. Letting $\rK(A)$ denote the Grothendieck group of the category of finitely generated $A$-modules, we have $\rK(A)=\bigoplus_{r=0}^d \rK(A)_r$, where $\rK(A)_r$ is the Grothendieck group of $\rD^b_{\rm fg}(A)_r$ (since we are interested in projective resolutions, we index homologically and bounded means bounded below). By \cite[Theorem 6.19]{symu1}, we have a natural isomorphism $\rK(A)_r = \Lambda \otimes \rK(\Gr_r(E))$, where $\Lambda$ is the ring of symmetric functions and $\Gr_r(E)$ is the Grassmannian of $r$-dimensional quotient spaces of $E$. We note that $\Theta$ defines an additive function on $\rK(A)$.

For a partition $\lambda$, we let $\lambda[n^r]$ be the partition $(n, \ldots, n, \lambda_1, \lambda_2, \ldots)$, where the first $r$ coordinates are $n$. This is a partition provided that $n \ge \lambda_1$. Given two partitions $\mu,\nu$, we say that $\mu$ is contained in $\nu$, and write $\mu \subseteq \nu$, if $\mu_i \le \nu_i$ for all $i$.

\begin{lemma} \label{lem:gamma-1}
Let $c \in \rK(A)_r$ be the class $s_{\lambda} \otimes [\cF]$, where $\cF$ is a coherent sheaf on $\Gr_r(E)$.
\begin{enumerate}
\item Every partition appearing in $\Theta_c$ is contained in $\lambda[n^r]$ for some $n$.
\item For $n \ge \lambda_1$, the coefficient of $\lambda[n^r]$ in $\Theta_c$ is $h_{\cF}(n)$, where $h_{\cF}$ is the Hilbert polynomial of $\cF$ with respect to the Pl\"ucker embedding.
\end{enumerate}
\end{lemma}

\begin{proof}
Let $\cQ$ be the rank $r$ tautological quotient bundle on $X=\Gr_r(E)$ and let $B=\Sym(\bV \otimes \cQ)$, which can be thought of as a tca on $X$. If $M$ is a $B$-module then $\Gamma(X, M)$ is naturally an $A$-module \cite[\S 6.2]{symu1}. Under the description of $\rK(A)$ given above, $c$ is the class of the complex $\rR \Gamma(X, M)$ where $M=\bS_{\lambda}(\bV) \otimes \cF \otimes B$ (see \cite[\S 6.6]{symu1}). Using the Cauchy decomposition for $B$, we have
\begin{displaymath}
\rH^i(X, M) = \bS_{\lambda}(\bV) \otimes \bigoplus_{\ell(\mu) \le r} \big( \bS_{\mu}(\bV) \otimes \rH^i(X, \cF \otimes \bS_{\mu}(\cQ)) \big).
\end{displaymath}
Note that the cohomology group above is just a vector space; the $\GL$ action comes from the first two Schur functors. Since $\mu$ has at most $r$ rows, the Littlewood--Richardson rule shows that all partitions appearing in $\bS_{\lambda} \otimes \bS_{\mu}$ are contained in $\lambda[n^r]$ for some $n$. This proves (a). The Littlewood--Richardson rule also shows that $\lambda[n^r]$ appears with multiplicity one in $\bS_{\lambda} \otimes \bS_{(n^r)}$ for $n \ge \lambda_1$, and does not appear in any other $\bS_{\lambda} \otimes \bS_{\mu}$ with $\ell(\mu) \le r$. Note that $\bS_{(n^r)}(\cQ) = \det(\cQ)^{\otimes n}$ and $\det(\cQ)$ is the Pl\"ucker bundle. We thus see that the coefficient of $\lambda[n^r]$ in $\Theta_c$ is
\begin{displaymath}
\sum_{i \ge 0} (-1)^i \dim \rH^i(X, \cF(n)) = h_{\cF}(n),
\end{displaymath}
which proves (b).
\end{proof}

\begin{proof}[Proof of Theorem~\ref{thm:gamma}]
Let $M$ be a finitely generated $A$-module, and suppose that $M$ is supported on $V(\fa_r)$ with $r$ minimal. By \cite[Theorem 6.19]{symu1}, we then have the following:
\begin{itemize}
\item In $\rK(A)$, we have $[M]=c_0+\cdots+c_r$ with $c_i \in \rK(A)_i$. Write $c_i=\sum_{\lambda} c_{i,\lambda}$ where $c_{i,\lambda} = s_{\lambda} \otimes [\cF_{i,\lambda}]$ and $\cF_{i,\lambda}$ is a coherent complex on $\Gr_i(E)$.
\item The class $[\cF_{r,\lambda}]$ is effective, i.e., we can assume $\cF_{r,\lambda}$ is a coherent sheaf.
\item There is a partition $\lambda$ such that $[\cF_{r,\lambda}] \ne 0$.
\end{itemize}
By Lemma~\ref{lem:gamma-1}(a) a partition with $\le n$ columns appearing with non-zero coefficient in $\Theta_{c_{i,\mu}}$ has size $\le in+\vert \mu \vert$. We thus see that $\gamma_M(n) \le rn+b$ where $b$ is the maximal size of a partition $\lambda$ with $\cF_{r,\lambda} \ne 0$, at least for $n \gg 0$.

Now, let $\lambda$ be a partition of size $b$ with $\cF_{r,\lambda}$ non-zero. By Lemma~\ref{lem:gamma-1}(b), $\lambda[n^r]$ appears with positive coefficient in $\Theta_{c_{r,\lambda}}$ for $n \gg 0$. Furthermore, the lemma shows that $\lambda[n^r]$ does not appear in $\Theta_{c_{i,\mu}}$ for any $(i,\mu) \ne (r,\lambda)$ and for $n \gg 0$. We thus see that $\lambda[n^r]$ has positive coefficient in $\Theta_M$, and so $\gamma_M(n) \ge rn+b$. This completes the proof.
\end{proof}

\begin{remark}
The proof shows that the slope of $\gamma_M$ is the minimal $r$ such that $M$ is supported on $V(\fa_r)$.
\end{remark}

\begin{remark} \label{rmk:elem}
Here is how one can prove Theorem~\ref{thm:gamma} without using the theory of \cite{symu1}. For a polynomial representation $M$, let $M[n]$ be the sum of the $\lambda$-isotypic pieces of $M$ over those $\lambda$ of size at least $n$ and with at most $n$ columns, and let $M^!=\bigoplus_{n \ge 0} M[n]$. Suppose $M$ is a finitely generated $A$-module. One then shows that $M^!$ is a finitely generated $A^!$-module, and from this deduces the structure of the bi-variate Hilbert series of $M^!$ (note that $M^!$ is bi-graded since each $M[n]$ is graded). One can deduce the theorem from this, as the Hilbert series determine $\gamma_M$.
\end{remark}

\section{Depth and projective dimension} \label{s:depth}

Let $M$ be an $A$-module. (We remind the reader that part of the definition of $A$-module is that that $M$ is a polynomial representation of $\GL$.) We write $\depth_M(n)$ or $\depth(M; n)$ for the depth of $M(\bC^n)$ as an $A(\bC^n)$-module, and $\pdim_M(n)$ or $\pdim(M; n)$ for the projective dimension of $M(\bC^n)$ as an $A(\bC^n)$-module. Our main result is the following theorem:

\begin{theorem} \label{thm:depth}
If $M$ is a finitely generated $A$-module then $\pdim_M$ and $\depth_M$ are eventually linear with slope at most $d$.
\end{theorem}

\begin{example}
Let $M=A/\fa_r$ be the coordinate ring of the rank $\le r$ matrices, as in Example~\ref{ex:gamma}. Suppose that $\min(n,d) \ge r$. Then $M(\bC^n)$ has codimension $(d-r)(n-r)$ and is Cohen--Macaulay, so its projective dimension is $\pdim_M(n) = (d-r)n - (d-r)r$. And by the Auslander--Buchsbaum formula, its depth is $\depth_M(n) = rn + r(d-r)$.
\end{example}

We now prove Theorem~\ref{thm:depth}. The Auslander--Buchsbaum formula states that
\[
  \depth_M(n) + \pdim_M(n) = dn,
\]
which allows us to deduce the result for $\depth$ from that for $\pdim$.

Using \cite[Theorem 7.7]{symu1}, there are finitely generated $A$-modules $F_k(M)$ that can be extracted from the linear strands of the minimal free resolution of $M$; its graded components are given by
\[
  F_k(M)_{p+k} = \Tor_p^A(M, \bC)_{p+k}^{\dagger, \vee},
\]
where $\vee$ is the duality on polynomial functors which fixes simple objects (see \cite[(6.1.6)]{expos}), and $\dagger$ is the equivalence on polynomial functors which interchanges the usual symmetric structure with the graded symmetric structure, and in particular has the effect $\bS_\lambda^\dagger = \bS_{\lambda^\dagger}$ (see \cite[(6.1.5)]{expos}). There are only finitely many values of $k$ for which $F_k(M)$ is non-zero.

The theorem is now a consequence of Theorem~\ref{thm:gamma} and the following lemma:

\begin{lemma} \label{lem:pdformula}
Let $M$ be a finitely generated $A$-module. Then
\begin{displaymath}
\pdim_M(n) = \max_k ( \gamma(F_k(M); n) - k ).
\end{displaymath}
\end{lemma}

\begin{proof}
Fix $n$, and let $N$ be the maximum appearing on the right side of the above equation. For this proof, write $T_i(M)$ for $\Tor^{A}_i(M, \bC)$. By definition, we have
\begin{displaymath}
T_p(M) = \bigoplus_k F_k(M)_{p+k}^{\dag,\vee}.
\end{displaymath}
We thus see that $T_q(M)(\bC^n) \ne 0$ for some $q \ge p$ if and only if there exists some $k$ such that $F_k(M)$ has a partition of size at least $p+k$ with at most $n$ columns, that is, $\gamma(F_k(M); n) \ge p+k$. Therefore, the maximum $p$ for which $T_p(M)(\bC^n) \ne 0$ is $p=N$, and the result follows since $\pdim_M(n)$ is the maximum $p$ for which
\begin{displaymath}
T_p(M)(\bC^n) = \Tor_p^{A(\bC^n)}(M(\bC^n),\bC)
\end{displaymath}
is non-zero.
\end{proof}

\section{Krull dimension}

Let $B$ be a quotient tca of $A$. Define $\delta_B(n)$ to be the Krull dimension of the ring $B(\bC^n)$. Since the defining ideal for $B$ is stable under the infinite symmetric group $\fS$, it follows from \cite[Theorem~7.10]{NagelRoemer} that $\delta_B$ is eventually linear. We now give an easy proof of a more precise result by leveraging the theory from \cite{symu1}.

We first recall some relevant information from \cite[\S 3]{symu1}. Let $C$ be any tca. An ideal $I$ of $C$ is {\bf prime} if, given any other ideals $J, J'$ of $C$, we have that $JJ' \subseteq I$ if and only if $J \subseteq I$ or $J' \subseteq I$. (Note that, by definition, all ideals are $\GL$-stable.) The {\bf spectrum} $\Spec(C)$ is defined to be the set of prime ideals of $C$, and is equipped with the Zariski topology (defined in the same way as for ordinary rings).

Next, let $\Gr_r(E)$ denote the underlying topological space of the Grassmannian (thought of as a scheme) parametrizing rank $r$ quotients of $E$. The {\bf total Grassmannian} of $E$, denoted $\Gr(E)$, is $\coprod_{r=0}^d \Gr_r(E)$ as a set. We topologize $\Gr(E)$ by defining a subset $Z \subset \Gr(E)$ to be closed if and only if
\begin{itemize}
\item $Z \cap \Gr_r(E)$ is closed for all $r$, and 
\item $Z$ is closed under taking quotients: if $E \to U$ is in $Z$, then so is $E \to U'$ for any quotient space $U'$ of $U$.
\end{itemize}
By \cite[Theorem 3.3]{symu1}, we have a homeomorphism $\Spec(A) \cong \Gr(E)$, and hence $\Spec(B)$ can be identified with a closed subset of $\Gr(E)$. If $Z \subset \Gr_r(E)$ is a Zariski closed irreducible subset, then its closure in $\Gr(E)$ is irreducible, and every irreducible closed subset of $\Gr(E)$ is of this form \cite[Proposition 3.2]{symu1}. Hence we can label irreducible closed subsets of $\Gr(E)$ by pairs $(r,Z)$ where $Z\subset \Gr_r(E)$ is a Zariski closed irreducible subset.

We then have the following result:

\begin{theorem}
Let $B$ be a quotient tca of $A$, and recall that $d=\dim(E)$.
\begin{enumerate}[\indent \rm (a)]
\item There exist integers $0 \le a \le d$ and $0 \le b \le (d-a)a$ such that $\delta_B(n)=an+b$ for all $n \gg 0$.
\end{enumerate}
Now assume that $\Spec(B)$ is irreducible.
\begin{enumerate}[\indent \rm (a)]
\setcounter{enumi}{1}
\item If $\Spec(B)$ corresponds to the pair $(r, Z)$, then $a=r$ and $b=\dim{Z}$.
\item If $b=0$ then $\Spec(B)=V(I)$ where $I$ is generated by linear forms.
\item If $b=(d-a)a$ then $\Spec(B)$ is the determinantal variety of rank $\le a$ maps.
\end{enumerate}
\end{theorem}

\begin{proof}
By noetherianity of $A$, $\Spec(B)$ has finitely many irreducible components, so it suffices to prove (a) when $\Spec(B)$ is irreducible. We will assume that from the beginning.  Suppose $\Spec(B)$ corresponds to $(r, Z)$.  Let $Y_n \subset \Spec(A(\bC^n))$ be the space of maps of rank exactly $r$.  Then the natural map $\pi_n \colon Y_n \to \Gr_r(E)$ is a fibration of relative dimension $rn$.  Furthermore, $\Spec(B(\bC^n))$ is the inverse image of $Z$ under $\pi_n$ (see \cite[Lemma~3.7]{symu1}). This proves (a) and (b).  If $b=0$ then $Z$ is a point, while if $b=(d-a)a$ then $Z$ is all of $\Gr_r(E)$; (c) and (d) follow.
\end{proof}

\end{document}